\documentclass{amsart}

\numberwithin{equation}{section}

\usepackage{amsmath}
\usepackage{amsfonts}
\usepackage{amsthm}
\usepackage{amssymb}
\usepackage{bm}
\usepackage[applemac]{inputenc}
\usepackage{enumerate}
\newtheoremstyle{mystyle}{}{}{\slshape}{2pt}{\scshape}{.}{ }{} 

\newtheorem{thm}{Theorem}[section]
\newtheorem{cor}[thm]{Corollary}

\newtheorem{prop}[thm]{Proposition}
\newtheorem{lemme}[thm]{Lemma}
\newtheorem{fait}[thm]{Fact}

\newtheorem{question}[thm]{Question}

\theoremstyle{definition}
\newtheorem{defi}[thm]{Definition}
\theoremstyle{mystyle}

\theoremstyle{remark}

\newcommand{\impl}{\rightarrow}

\DeclareMathOperator{\Av}{Av}

\def\indsym#1#2{%
 \setbox0=\hbox{$\m@th#1x$}%
 \kern\wd0%
 \hbox to 0pt{\hss$\m@th#1\mid$\hbox to 0pt{$\m@th#1^{#2}$\hss}\hss}%
 \lower.9\ht0\hbox to 0pt{\hss$\m@th#1\smile$\hss}%
 \kern\wd0}

\def\nindsym#1#2{%
 \setbox0=\hbox{$\m@th#1x$}%
 \kern\wd0%
 \hbox to 0pt{\hss$\m@th#1\not$\kern1.4\wd0\hss}
 \hbox to 0pt{\hss$\m@th#1\mid$\hbox to 0pt{$\m@th#1^{#2}$\hss}\hss}%
 \lower.9\ht0\hbox to 0pt{\hss$\m@th#1\smile$\hss}%
 \kern\wd0}

\title{A Note on ``Regularity lemma for distal structures"}
\author{Pierre Simon}
\address{Universit\'e Claude Bernard - Lyon 1\\
Institut Camille Jordan\\
43 boulevard du 11 novembre 1918\\
69622 Villeurbanne Cedex, France}
\email{simon@math.univ-lyon1.fr}
\thanks{Partially supported by ValCoMo (ANR-13-BS01-0006).}

\date{\today}

\begin{document}
\maketitle

In their recent paper \cite{CherStar}, Chernikov and Starchenko prove that graphs defined in distal theories have the \emph{strong Erd\H os-Hajnal property}: if $R\subseteq M^m \times M^n$ is a definable relation in some distal structure $M$, there is $\alpha>0$ such that if $A\subseteq M^m$, $B\subseteq M^n$ are finite sets, then there are $A_0\subseteq A$, $B_0\subseteq B$ with $|A_0|\geq \alpha |A|$, $|B_0|\geq \alpha |B|$ and $R\cap A_0\times B_0$ is either empty or equal to $A_0 \times B_0$. In fact, they prove a more general statement for hypergraphs and generically stable measures instead of counting measure, see Theorem \ref{th_main} below. This generalizes a theorem of Alon et al. \cite{APPRS} who proved it if $R$ is semi-algebraic.

Chernikov and Starchenko's proof uses both the theory of Keisler measures in NIP and some combinatorial arguments. The purpose of this note is to remove the latter and give a purely model-theoretic (and in fact quite short) proof of the strong Erd\H os-Hajnal property in distal theories. We also answer a small question raised in \cite{CherStar} about uniformly cutting finite sets and show that this property is equivalent to having no generically stable types. We keep this note short and refer to the excellent introduction of \cite{CherStar} for background.

\smallskip
Thanks to Sergei Starchenko for comments on a preliminary draft.

\section{Generically stable and smooth measures}

We record here some facts about generically stable and smooth measures in NIP theories. Everything appeared in previous works, possibly in different terms. See e.g. \cite[Chapter 7]{NIPBook} for details.

Throughout this note, we assume that our ambient theory $T$ is NIP. Recall that a (Keisler) measure $\mu(x)$ over some set $A$ is a finitely additive probability measure on the boolean algebra of $A$-definable sets (in free variable $x$). Such a measure extends uniquely to a regular Borel measure on $S_x(A)$.

Of special importance are measures $\mu(x)$ which can be approximated by average measures on finite sets. Such measures are called \emph{generically stable}. Here is one way to define them, where $\Av(a_1,\ldots,a_n;X)$ means $\frac 1 n |\{i : a_i \in X\}|$.

\begin{defi}\label{defi_genstable}
A measure $\mu(x)$ over a model $M$ is generically stable if for every formula $\phi(x;y)$ and $\epsilon$, there are $a_1,\ldots,a_n \in M$ such that for any $b\in M^{|y|}$, $|\mu(\phi(x;b))-\Av(a_1,\ldots,a_n;\phi(x;b))|\leq \epsilon$.\end{defi}

In particular, the normalized counting measure on a finite set $\{a_1,\ldots,a_n\}$ is generically stable. The VC-theorem implies that the number $n$ above depends only on $\phi(x;y)$ and $\epsilon$:

\begin{fait}\label{fact_genstable}
Let $\phi(x;y)$ be a formula and $\epsilon>0$, then there is an integer $n$ such that for any generically stable measure $\mu(x)$ on a model $M$, they are $a_1,\ldots,a_n \in M$ such that for any $b\in M^{|y|}$, $|\mu(\phi(x;b))-\Av(a_1,\ldots,a_n;\phi(x;b))|\leq \epsilon$.
\end{fait}

In fact, the number $n$ depends only on the VC-dimension of $\phi$ (and $\epsilon$).

\begin{cor}\label{cor_ultra}
Let $I$ be an index set and for $i\in I$, $\mu_i(x)$ a generically stable measure on a model $M_i$. For $\mathcal U$ an ultrafilter on $I$, the ultraproduct $\prod_{\mathcal U} \mu_i$ is a generically stable measure on $\prod_{\mathcal U} M_i$.
\end{cor}
\begin{proof}
Let $\tilde \mu = \prod_{\mathcal U} \mu_i$ and fix a formula $\phi(x;y)$ and $\epsilon>0$. For each $i$, take $a_1(i),\ldots, a_n(i)$ in $M_i$ given by Fact \ref{fact_genstable} for $\mu_i$. Define $\bar a_k := [a_k(i):i\in I]\in \prod_\mathcal U M_i$. Then for any parameter $b$, we have $|\tilde \mu (\phi(x;b)) - \Av(\bar a_1,\ldots,\bar a_k;\phi(x;b))|\leq \epsilon$. Hence $\tilde \mu$ is generically stable.
\end{proof}

%Therefore generically stable measures are exactly ultraproducts of finite counting measures.

%\smallskip
A measure $\mu(x)$ over $M$ is called \emph{smooth} if it has a unique extension to any bigger model $N$. We have the following equivalent condition (which follows easily by compactness, see \cite[Lemma 7.8]{NIPBook}).

\begin{fait}\label{fact_smooth}
A measure $\mu(x)$ over $M$ is smooth if and only if for every formula $\phi(x;y)$ and $\epsilon>0$, there are finitely many formulas $\psi_i(y)$ and $\theta^1_i(x)$, $\theta^2_i(x)$ all over $M$ such that:
\begin{enumerate}
\item the formulas $\psi_i(y)$ partition $y$-space;

\item for all $i$ and $b\in M^{|y|}$, if $M\models \psi_i(b)$, then
\[
M\models (\theta_i^1(x) \impl \phi(x;b) \impl \theta_i^ 2(x));
\]

\item for all $i$, $\mu(\theta_i^2(x))-\mu(\theta_i^ 1(x))\leq \epsilon$.
\end{enumerate}
\end{fait}

Let $\mu(x)$ and $\nu(y)$ be two measures over the same base $M$. We say that a measure $\omega(x,y)$ over $M$ is a \emph{product measure} of $\mu(x)$ and $\lambda(y)$ if $\omega(\phi(x)\wedge \psi(y))=\mu(\phi(x))\cdot \lambda(\psi(y))$ for any two definable sets $\phi(x)$ and $\psi(y)$. We let $(\mu \times \lambda)(x,y)$ denote the partial measure defined on the boolean algebra generated by rectangles $\phi(x)\wedge \psi(y)$ and giving such a set the product measure $\mu(\phi(x)) \cdot \lambda(\psi(y))$ (this determines the measure on the boolean algebra since a boolean combination of rectangles can be written as a disjoint union of rectangles). Note that if $\lambda(y)$ is a type, then any extension of $\mu(x) \cup \lambda(y)$ is a product measure.

We say that two measures $\mu(x)$ and $\lambda(y)$ over $M$ are \emph{weakly orthogonal} if there is a unique product measure of $\mu(x)$ and $\lambda(y)$.

\begin{lemme}
Two measures $\mu(x)$ and $\lambda(y)$ over $M$ are weakly orthogonal if and only if the following property holds:

For any $M$-definable set $D(x,y)$ and $\epsilon>0$, there are two definable sets $D^-_{\epsilon}(x,y)$ and $D^+_\epsilon(x,y)$ such that:

$(i)$ $D^-_\epsilon$ and $D^+_\epsilon$ are union of rectangles $A_i(x)\times B_i(y)$;

$(ii)$ $D^-_\epsilon \subseteq D \subseteq D^+_\epsilon$;

$(iii)$ $(\mu \times \lambda)(D^+_\epsilon) - (\mu\times \lambda)(D^-_\epsilon) \leq \epsilon$.
\end{lemme}
\begin{proof}
Right to left: the conditions imply that any product measure must give $D$ measure $\inf_{\epsilon>0} (\mu \times \lambda)(D^+_\epsilon) = \sup_{\epsilon <0}(\mu \times \lambda)(D^-_\epsilon)$. Hence there is a unique such measure.

For the converse, we use Lemma 7.3 in \cite{NIPBook}. It says that given a boolean algebra $\Omega$ of definable sets and a partial measure $\mu_0$ on $\Omega$, we can extend $\mu_0$ to a Keisler measure $\mu$ and additionally impose $\mu(D)=r$ for some definable set $D$ and real $r\in [0,1]$ as long as there is no obvious obstruction, namely any $B\in\Omega$, $B\subseteq D$ has measure $\leq r$ and any $B\in \Omega$, $D\subseteq B$ has measure $\geq r$. We apply this here with $\Omega$ being the algebra generated by rectangles.
\end{proof}

\begin{lemme}
Let $\mu(x)$ be a measure over $M$. Then $\mu(x)$ is smooth if and only if it is weakly orthogonal to all measures over $M$.
\end{lemme}
\begin{proof}
By definition, a measure is smooth if and only if it is weakly orthogonal to all types over $M$. Hence right to left is fine.

Left to right is \cite[Corollary 2.5]{NIP3} and follows easily from Fact \ref{fact_smooth}: a formula $\phi(x;y)$ is well approximated from below by the union of rectangles $P(x,y):=\bigvee_{i} \theta^0_i(x)\wedge \psi_i(y)$ and from above by $Q(x,y):=\bigvee_i \theta^1_i(x) \wedge \psi_i(y)$. If $\omega(x,y)$ is a product measure of $\mu(x)$ and some $\lambda(y)$, then $\omega(Q)-\omega(P)$ cannot be more than $\epsilon$.
\end{proof}

If $R(x_1,\ldots,x_n)$ is a relation and sets $A_i\subseteq M^{|x_i|}$ are given, we say that the tuple $(A_1,\ldots,A_n)$ is $R$-homogeneous if either for all $(x_1,\ldots,x_n)\in A_1\times \cdots \times A_n$, $R(x_1,\ldots,x_n)$ holds or  for all $(x_1,\ldots,x_n)\in A_1\times \cdots \times A_n$, $\neg R(x_1,\ldots,x_n)$ holds.

\begin{cor}\label{cor_smoothdec}
Let $\mu(x)$ be a smooth measure and $\lambda(y)$ any Keisler measure. Let $R(x;y)$ be a formula over $M$ and $\epsilon>0$. Then there is an $M$-definable partition $P_1,\ldots,P_n$ of $x$-space and $Q_1,\ldots,Q_m$ of $y$-space such that 
\[
\sum \mu(P_i)\lambda(Q_j) < \epsilon,
\]
where the sum runs over all pairs $(i,j)$ such that $(P_i,Q_j)$ is not $R$-homogenous.
\end{cor}
\begin{proof}
By the previous lemmas, there are two definable sets $P, Q\subseteq M^{|x|+|y|}$ which are union of rectangles $A_i(x)\times B_i(y)$ such that:

1. $P\subseteq R \subseteq Q$;

2. $\omega(Q)-\omega(P)<\epsilon$, for any product measure $\omega(x,y)$ of $\mu(x)$ and $\lambda(y)$.

We then obtain the partition $(P_i)_i$ by taking all the atoms in the boolean algebra generated by the $A_i$'s and same for $(Q_j)_j$ and the $B_i$'s.
\end{proof}

This generalizes immediately to a product of $n$-many smooth measures (or indeed $n-1$ smooth measures and an arbitrary one).

\begin{cor}\label{cor_smoothdec2}
Let $\mu_1(x_1),\ldots,\mu_n(x_n)$ be smooth measures over $M$. Fix a formula $R(x_1,\ldots,x_n)$ and $\epsilon>0$. Then for each $i$, there is an $M$-definable partition $P(i,1),\ldots,P(i,m(i))$ of $x_i$-space such that
\[
\sum \mu_1(P(1,i_i))\cdots \mu_n(P(n,i_n)) < \epsilon,
\]
where the sum runs over all tuples $(i_1,\ldots,i_n)$ for which the family $(P(1,i_1),\ldots,P(n,i_n))$ is not $R$-homogeneous.
\end{cor}

\section{Distality and the main result}

The class of distal theories was introduced in \cite{distal} to capture the notion of an order-like, or purely unstable, NIP theory. There are various equivalent definitions of distality in terms of indiscernible sequences, invariant types, measures, arbitrary types... We only need to know here that a theory is distal if and only if all generically stable measures are smooth. Thus we have:

\begin{lemme}\label{lem_ultradistal}
Let $T$ be distal, then an ultraproduct of smooth measures is smooth.
\end{lemme}
\begin{proof}
This follows immediately from Corollary \ref{cor_ultra} and the fact that generically stable measures and smooth measures coincide.
\end{proof}

We now have all we need to prove the main theorem. The following is Corollary 4.6 in \cite{CherStar}, which is the final statement in Sections 3 and 4 of that paper. We prove the version with parameters directly, although as observed in the proof of Corollary 4.6, it would follow at once from the simpler parameter-free version.

\begin{thm}\label{th_main}
Let $T$ be distal, $R_y(x_1,\ldots,x_d)=R(x_1,\ldots,x_d;y)$ a definable relation. Given $\alpha>0$, there is $\delta>0$ such that: for any parameter $b$, for any generically stable (equiv. smooth) product measure $\omega$ on $M^{|x_1|}\times \cdots \times M^{|x_k|}$ with $\omega(R_b)\geq\alpha$, there are definable sets $A_i \subseteq M^{|x_i|}$ with $\omega|_{x_i} (A_i)> \delta$ for all $i$ such that $\prod_i A_i \subseteq R_b$.

Moreover, each $A_i$ is defined by an instance of a formula that depends only on $R$ and $\alpha$.
\end{thm}
\begin{proof}
This is an direct consequence of Corollary \ref{cor_smoothdec} and Lemma \ref{lem_ultradistal}. We give details. First, by Corollary \ref{cor_smoothdec2}, given $\omega$ and $R_b$, we can find sets $A_i \subseteq M^{|x_i|}$ such that $\prod A_i\subseteq R_b$ and $\omega|_{x_i}(A_i)>\delta$ for some $\delta>0$ depending on all the data.

Now assume that the conclusion is false. For simplicity of notations, assume $L$ is countable and list all tuples of the form $(\theta_1(x_1;y_1),\ldots,\theta_d(x_d;y_d))$ as $(\bar \theta_j : j<\omega)$, where $\bar \theta_j = (\theta^j_1(x_1;y_{1,j}),\ldots,\theta^j_d(x_d;y_{d,j}))$. Given an integer $n$, there is a smooth measure $\omega=\omega(n)$ over a model $M=M(n)$ and a parameter $b=b(n)\in M(n)$ such that $\omega(R_b)\geq \alpha$ and for any $j\leq n$ and any choice of parameters $b_{i,j}$ from $M$, setting $A_i = \theta_i^j(M;b_{i,j})$, either for some $i$, $\omega|_{x_i}(A_i)\leq1/n$, or it is not the case that $\prod_i A_i \subseteq R_b$.

Take a non-principal ultraproduct of the $\omega(n)$ to obtain a measure $\tilde \omega$ over some model $\tilde M$ and some $\tilde b\in \tilde M$ such that $\tilde \omega(R_{\tilde b})\geq \alpha$. By Lemma \ref{lem_ultradistal}, $\tilde \omega$ is smooth. Hence by the first paragraph, we can find $A_i$'s and $\delta$ for $\tilde \omega$ and $R_{\tilde b}$ as in the statement. The same $\delta$ and formulas defining the $A_i$'s will work for almost all factors, contradicting the construction.
\end{proof}

We can also prove the asymmetric version which is Theorem 3.6 in \cite{CherStar}.
\begin{thm}\label{th_asym}
Let $T$ be distal and $R(x,y)$ a definable relation. Let $\beta \in (0,\frac 1 2)$. There are $\alpha\in (0,1)$ and formulas $\phi_1(x;z_1), \phi_2(y;z_2)$ such that:

For any Keisler measure $\lambda(x)$ and generically stable measure $\mu(y)$ both on $M$, there are parameters $c_1\in M^{|z_1|}$ and $c_2\in M^{|z_2|}$ with $\lambda(\phi_1(x;c_1))\geq \alpha$, $\mu(\phi_2(y;c_2))\geq \beta$ and the pair $\phi_1(M;c_1)$, $\phi_2(M;c_2)$ is $R$-homogeneous.
\end{thm}
\begin{proof}
Choose $R$ and $\beta\in (0, \frac 1 2 - \epsilon)$ and fix a pair of measures $\lambda(x)$, $\mu(y)$ as in the statement. As $\mu$ is smooth, we can apply Fact \ref{fact_smooth} to it, with parameter $\epsilon$. It gives us formulas $\psi_i(x)$ and $\theta_i^k(y)$. Take $i$ such that $\lambda(\psi_i(x))=:\alpha>0$ and set $\phi_1(x)=\psi_i(x)$. At least one of $\theta_i^1(y)$ or $\neg\theta_i^2(y)$ has measure $\geq \beta$. Let $\phi_2(y)$ be equal to it. Then the pair $(\phi_1(x),\phi_2(y))$ (with hidden parameters from $M$) is as required for the given pair of measures. We conclude by a compactness argument as in Theorem \ref{th_main}.
\end{proof}
%So the main point here, is that if the theory is distal, then formulas in Fact \ref{fact_smooth} (or in Corollary \ref{cor_smoothdec}) can be chosen uniformly in $\mu$.

\section{Equipartitions}

In Remark 5.11 of \cite{CherStar}, it is asked whether in distal theories, one can cut finite sets uniformly. We answer this question positively.

\begin{defi}\label{def_cut}
Let $\mathcal S$ be a sort. We say that $T$ uniformly cuts finite sets in $\mathcal S$ if for every $\epsilon>0$, there is a formula $\chi(x;y)$ such that for any sufficiently large finite set $A$ in $\mathcal S$, and $r\in [0,1]$, there is a parameter $b$ for which
\[
\left | \frac {|\chi(A;b)|}{|A|} - r\right | \leq \epsilon.
\]
\end{defi}

We will say that $T$ \emph{uniformly cuts generically stable measures} on $\mathcal S$ if the conclusion of Proposition 5.12 in \cite{CherStar} holds, namely: For every formula $\phi(x;y)$ and $\epsilon>0$, there is some $\chi(x;z)$ such that for any generically stable measure $\mu$ on $M$ with $\mu(\{c\})=0$ for any singleton $c\in M^{|x|}$, if $0\leq r\leq \mu(\phi(x;a))$, \underline{then} we can find $b\in M$ with $|\mu(\phi(x;a)\cap \chi(x;b)) -r |\leq \epsilon$.

In the following proof, we write $a \approx_\epsilon b$ for $|a-b|\leq \epsilon$.
\begin{lemme}
Let $\mu(x)$ be generically stable over $M$ and $p\in S_x(M)$ such that $\mu(\{p\})>0$. Then $p$ is generically stable.
\end{lemme}
\begin{proof}
This can be seen in various ways. Here is an argument which does not use any additional fact about generically stable measures. Fix a formula $\phi(x;y)$ and $\epsilon>0$. Set $\alpha = \mu(\{p\})$. By regularity of $\mu$ seen as a measure on $S_x(M)$, there is a formula $\theta(x)\in L(M)$ such that $p\vdash \theta(x)$ and $\mu(\theta(x)) \leq (1+\epsilon)\alpha$. Set $\phi'(x;y) = \phi(x;y)\wedge \theta(x)$ and without loss, assume that for some $b_0$, $\phi'(x;b_0) = \theta(x)$. Take $a_1,\ldots,a_n$ given by Definition \ref{defi_genstable} for $\mu$, $\phi'$ and $\epsilon':=\epsilon \alpha$. Reordering, there is $k\leq n$ such that $a_i \models \theta(x)$ if and only if $i\leq k$. We have $|k/n|=(1+\delta) \alpha$ for $\delta\leq2\epsilon$. Then for any $b$, $p(\phi(x;b)) \approx_\epsilon \frac 1 \alpha \mu(\phi'(x;b))  \approx_\epsilon \frac 1 \alpha \Av(a_1,\ldots,a_n;\phi'(x;b)) = \Av(a_1,\ldots,a_k;\phi(x;b))\frac k n \frac 1 \alpha \approx_\delta \Av(a_1,\ldots,a_k;\phi(x;b))$. Hence $p$ is generically stable.
\end{proof}

\begin{prop}
Let $T$ be NIP and $\mathcal S$ a sort. The following are equivalent:

(1) $T$ uniformly cuts finite sets in $\mathcal S$;

(2) $T$ uniformly cuts generically stable measures on $\mathcal S$;

(3) any generically stable type concentrating on $\mathcal S$ is realized.
\end{prop}
\begin{proof}
(3) $\Rightarrow$ (2): First, let $\mu(x)$ be a fixed generically stable measure over $M$ with $\mu(\{c\})=0$ for all $c$. Then if $\mu(\{p\})>0$ for some $p\in S_x(M)$, $p$ is generically stable. By (3) and the assumption on $\mu$, this does not happen. But then by regularity of $\mu$, for any $p$ in the support of $\mu$, we can find a clopen set $U_p\subseteq S_x(M)$ containing $p$ of measure $\mu(U_p)<\epsilon$. By compactness of the support, we can extract a finite cover which we can refine to be composed of disjoint sets. Hence we obtain a finite partition of $x$-space into definable sets of measures $\leq \epsilon$. By Corollary \ref{cor_ultra}, the size and the formulas involved in this partition can in fact be chosen uniformly in $\mu$. By standard coding techniques, we construct a formula $\chi(x;z)$ such that any finite union of members of this partition is equal to some $\chi(x;b)$. Then $\chi(x;z)$ has the required properties (and does not depend on $\phi(x;y)$).

(2) $\Rightarrow$ (1): This is a simple compactness argument. If (1) does not hold, then for a given $\epsilon$, there is a sequence $A_n$ of finite sets, $|A_n|\geq n$ such that for any formula $\chi(x;y)$ , for $n$ large enough, $\chi(x;y)$ cannot be used to cut $A_n$ with precision $\epsilon$ as in Definition \ref{def_cut}. Let $\mu_n$ be the normalized counting measure on $A_n$ and let $\bar \mu$ be a non-principal ultraproduct of the $\mu_n$'s. Then $\bar \mu$ is generically stable and $\bar \mu(\{c\})=0$ for all $c$. Assumption (2) applied to $\phi(x;a)= \top$ gives us a formula $\chi(x;y)$ which can be used to partition the space in sets of arbitrary $\bar \mu$-measures up to $\epsilon$. But then the same $\chi(x;y)$ works for almost all the measures $\mu_n$ contradicting the assumption.

(1) $\Rightarrow$ (3): Assume that (3) fails. Then there is a non-constant totally indiscernible sequence $I=(a_i:i<\omega)$ of elements of $\mathcal S$ (take a Morley sequence of a generically stable type). Then by NIP, for any formula $\chi(x;y)$ there is $N$ such that for any $b$, either $\chi(I;b)$ or $I\setminus \chi(I;b)$ has size $\leq N$. Thus taking large finite sets of $I$, we see that (1) cannot hold.
\end{proof}

Since distal theories satisfy (3), this gives a proof of \cite[Proposition 5.12]{CherStar} for all distal theories (that proposition asserts (2) above assuming distality and (1)).

We end with a question. It is proved in \cite{CherStar} that there is no infinite distal field of characteristic $p$. In general, it would be very interesting to understand what are the possible obstruction to an NIP theory having a distal expansion. One can also ask about expansions that uniformly cut finite sets.

\begin{question}
Does ACF$_p$ have an NIP expansion which uniformly cuts finite sets?
\end{question}

\bibliographystyle{alpha}
\bibliography{tout}

\end{document}